\documentclass[11pt]{article}

\usepackage{color}
\usepackage{tikz-cd}
\usepackage[colorlinks=true,linkcolor=blue,citecolor=blue]{hyperref}
\usepackage{amssymb}
\usepackage{amsmath}
\usepackage{amsthm}
\usepackage[normalem]{ulem}

\numberwithin{equation}{section}

  \newtheorem{THM}{Theorem}[section]
  \newtheorem{LEM}[THM]{Lemma}
  \newtheorem{PROP}[THM]{Proposition}
  \newtheorem{COR}[THM]{Corollary}

\renewcommand{\le}{\leqslant}
\renewcommand{\ge}{\geqslant}

\newcommand{\0}{\varnothing}

\renewcommand{\sec}{\cap}
\renewcommand{\phi}{\varphi}
\renewcommand{\epsilon}{\varepsilon}
\newcommand{\UNION}{\bigcup}
\newcommand{\AAA}{\mathbf{A}}

\newcommand{\CC}{\mathbf{C}}
\newcommand{\CCfin}{\CC_{\mathit{fin}}}
\newcommand{\DDfin}{\DD_{\mathit{fin}}}
\newcommand{\DD}{\mathbf{D}}

\newcommand{\KK}{\mathbf{K}}

\newcommand{\NN}{\mathbb{N}}

\renewcommand{\SS}{\mathbf{S}}

\newcommand{\restr}[2]{\hbox{$#1$}\hbox{$\upharpoonright$}_{#2}}
\newcommand{\Boxed}[1]{\mbox{$#1$}}

\newcommand{\id}{\mathrm{id}}

\newcommand{\ID}{\mathrm{ID}}
\newcommand{\Ob}{\mathrm{Ob}}

\newcommand{\Const}{\mathrm{Const}}

\newcommand{\op}{\mathrm{op}}

\newcommand{\Age}{\mathrm{Age}}

\newcommand{\iso}{\mathrm{iso}}

\newcommand\toCC{\overset{\CC}\longrightarrow}

\newcommand\toSS{\overset{\SS}\longrightarrow}

\newcommand\toCCC{\overset{\CC^*}\longrightarrow}

\newcommand{\calA}{\mathcal{A}}
\newcommand{\calB}{\mathcal{B}}
\newcommand{\calC}{\mathcal{C}}
\newcommand{\calD}{\mathcal{D}}

\newcommand{\calS}{\mathcal{S}}

\newcommand{\Fraisse}{Fra\"\i ss\'e}

\newcommand{\Aut}{\mathrm{Aut}}

\title{From Ramsey degrees to Ramsey expansions via weak amalgamation}
\author{%
  Dragan Ma\v sulovi\'c\\
  Department of Mathematics and Informatics\\
  Faculty of Sciences, University of Novi Sad, Serbia\\
  Trg Dositeja Obradovi\'ca 3, 21000 Novi Sad, Serbia\\
  \texttt{dragan.masulovic@dmi.uns.ac.rs}
\and
  Andy Zucker\\
  Department of Pure Mathematics\\
  University of Waterloo\\
  200 University Ave W, Waterloo, ON  N2L 3G1, Canada\\
  \texttt{a3zucker@uwaterloo.ca}
}

\begin{document}
\maketitle

\begin{abstract}
  Under no additional assumptions, in this paper we construct a Ramsey expansion
  for every category of finite objects with finite small Ramsey degrees.
  Our construction is based on the relationship between small Ramsey degrees, weak amalgamation and recent results about
  weak \Fraisse\ categories. Namely, generalizing the fact that every Ramsey class has amalgamation, we show that
  classes with finite Ramsey degrees have weak amalgamation. We then invoke
  the machinery of weak \Fraisse\ categories to perform the construction.
  This improves previous similar results where an analogous construction was
  carried out under the assumption that everything sits comfortably in a bigger class with enough infrastructure,
  and that in this wider context there is an ultrahomogeneous structure under whose umbrella the construction takes place.

  \medskip


  \medskip

\end{abstract}

\section{Introduction}

The intimate connection between the amalgamation property and the Ramsey property has been evident from the
very beginnings of Structural Ramsey Theory in early 1970's: amalgamation lies at the heart of the \emph{partite method},
one of the most powerful tools for proving the Ramsey property for a class of finite structures
(see, e.g.\ the proof of the famous Ne\v set\v ril-R\"odl Theorem~\cite{Nesetril-Rodl,Nesetril-Rodl-1989}).
On the other hand, it is an easy by fundamental result of Ne\v set\v ril and R\"odl from 1977
that every class of structures with the Ramsey property has the amalgamation property~\cite{Nesetril-Rodl}.

Many natural classes of finite structures with the amalgamation property (such as finite graphs and finite partially
ordered sets) do not enjoy the Ramsey property. It is quite common, though, that the structures can be expanded by
a few carefully chosen relations so that the resulting class of expanded structures has the Ramsey property.
We then say that the class of structures has a \emph{Ramsey expansion}.

In the late 1990s it was observed that many concrete classes of finite structures where a Ramsey expansion had been identified
also enjoyed a weaker property of having \emph{finite Ramsey degrees}~\cite{fouche97,fouche98,fouche99}. For \Fraisse\ classes,
that this is not a mere coincidence was proven in one direction in \cite{KPT}, who prove that classes with a Ramsey expansion have finite Ramsey degrees, and in the other by the second author in~\cite{zucker1}, showing that small Ramsey degrees suffice for the existence of a Ramsey expansion.
A more combinatorial proof of the same fact
can be found in~\cite{vanthe-finramdeg}, and a reinterpretation of the latter proof in the language of
category theory in~\cite{masul-kpt}. All these proofs make key use of the fact that the classes are \Fraisse, in the sense
that everything sits comfortably in a bigger class with enough infrastructure,
and that in this wider context there is an ultrahomogeneous structure under whose umbrella the construction takes place.

In this paper we show that the assumption about the ambient class in which an ultrahomogeneous object oversees the
construction is unnecessary. Imposing no additional assumptions, for each category of finite objects with finite small
Ramsey degrees (the definitions are given below) we construct a Ramsey expansion. The expansion is not constructed from
scratch, though. Section~\ref{srd-wap.sec.weak-fraisse-cats} starts with our first insight,
where we show (Theorem~\ref{srd-wap.thm.srd=>wap}) that classes with finite Ramsey degrees
have the weak amalgamation property. This is an analogue of Ne\v set\v ril and R\"odl's 
result~\cite{Nesetril-Rodl} that the Ramsey property implies amalgamation. In the remainder of the section we recall
from~\cite{kubis-weak-fraisse-cat,kubis-fraisse-seq} how an ambient category and a
weakly homogeneous object in it can be constructed from a category with weak amalgamation. 
In Section~\ref{srd-wap.sec.weak-hom-precomp-exp} we then upgrade the results from~\cite{masul-kpt} to show that
if everything sits in a bigger category in which there is a weakly homogeneous and locally finite object universal
for the category we are trying to expand, then there is a convenient expansion which can be trimmed down to a Ramsey expansion.
Finally, in Section~\ref{srd-wap.sec.main} we put all the ingredients together to prove
the main result of the paper. As a corollary, we specialize the main result of the paper to
arbitrary first-order structures, and then give a dual result about the relationship of small dual Ramsey degrees
and dual Ramsey expansions. 

\section{Preliminaries}
\label{srd-wap.sec.prelim}

Let us quickly fix some notation and conventions. All the categories in this paper are locally small.
Let $\CC$ be a category. By $\Ob(\CC)$ we denote the class of all the objects in $\CC$.
Hom-sets in $\CC$ will be denoted by $\hom_\CC(A, B)$, or simply $\hom(A, B)$ when $\CC$ is clear from the context.
The identity morphism will be denoted by $\id_A$ and the composition of morphisms by $\Boxed\cdot$ (dot).
If $\hom_\CC(A, B) \ne \0$ we write $A \toCC B$. 
A morphism $f \in \hom_\CC(B, C)$ is \emph{mono} if it is left cancellable, that is, $f \cdot g = f \cdot h \Rightarrow g = h$
for all $A \in \Ob(\CC)$ and all $g, h \in \hom_\CC(A, B)$. Dually, a morphism is \emph{epi} if it is right cancellable.
A category $\CC$ is \emph{directed} if for all $A, B \in \Ob(\CC)$
there is a $C \in \Ob(\CC)$ such that $A \toCC C$ and $B \toCC C$. Let $\iso_\CC(A, B)$ denote the set of all the
invertible morphisms from $\hom_\CC(A, B)$, and let $\Aut_\CC(A) = \iso_\CC(A, A)$ denote the set of all the
\emph{automorphisms of $A$}. We write $A \cong B$ to denote that $\iso_\CC(A, B) \ne \0$.
A \emph{skeleton} of a category $\CC$ is a full subcategory $\SS$ of $\CC$ such that
no two objects in $\SS$ are isomorphic and for every $C \in \Ob(\CC)$ there is an $S \in \Ob(\SS)$ such that
$C \cong S$. 
A subcategory $\DD$ of $\CC$ is \emph{cofinal in $\CC$} if
for every $C \in \Ob(\CC)$ there is a $D \in \Ob(\DD)$ such that $C \toCC D$.
As usual, $\CC^\op$ denotes the opposite category. Whenever the category $\CC$ is fixed we shall simply write $\hom(A, B)$, $\Aut(C)$, $t(A)$, etc.

For $k \in \NN$, a $k$-coloring of a set $S$ is any mapping $\chi : S \to k$, where,
as usual, we identify $k$ with $\{0, 1,\ldots, k-1\}$.
For positive integers $k, t \in \NN$ and objects $A, B, C \in \Ob(\CC)$ such that $A \toCC B$ we write
$
  C \longrightarrow (B)^{A}_{k, t}
$
to denote that for every $k$-coloring $\chi : \hom(A, C) \to k$ there is a morphism
$w \in \hom(B, C)$ such that $|\chi(w \cdot \hom(A, B))| \le t$.
(For a set of morphisms $F$ we let $w \cdot F = \{ w \cdot f : f \in F \}$.)
In case $t = 1$ we write
$
  C \longrightarrow (B)^{A}_{k}.
$

A category $\CC$ has the \emph{Ramsey property} if
for every integer $k \in \NN$ and all $A, B \in \Ob(\CC)$ there is a
$C \in \Ob(\CC)$ such that $C \longrightarrow (B)^{A}_k$.

For $A \in \Ob(\CC)$ let $t_\CC(A)$, the \emph{Ramsey degree of $A$ in $\CC$} (sometimes called the \emph{small} Ramsey degree)
 denotes the least positive integer $n$ such that for all $k \in \NN$ and all $B \in \Ob(\CC)$
 there exists a $C \in \Ob(\CC)$ such that $C \longrightarrow (B)^{A}_{k, n}$, if such an integer exists.
Otherwise put $t_\CC(A) = \infty$. A category $\CC$ has \emph{finite small Ramsey degrees}
if $t_\CC(A) < \infty$ for all $A \in \Ob(\CC)$.

As in~\cite{masul-kpt} we shall be working in the following setup. Let $\CC$ be a locally small category and
let $\CCfin$ be a full subcategory of $\CC$ such that the following holds:
\begin{description}
\item[(C1)]
   all the morphisms in $\CC$ are mono;
\item[(C2)]
  $\Ob(\CCfin)$ is a set;
\item[(C3)]
  for all $A, B \in \Ob(\CCfin)$ the set $\hom(A, B)$ is finite;
\item[(C4)]
  for every $F \in \Ob(\CC)$ there is an $A \in \Ob(\CCfin)$ such that $A \to F$; and
\item[(C5)]
  for every $B \in \Ob(\CCfin)$ the set $\{A \in \Ob(\CCfin) : A \to B\}$ is finite.
\end{description}
We think of objects in $\CCfin$ as templates of finite objects in $\CC$.
For the remainder of the section let us fix a locally small category $\CC$ and its full subcategory $\CCfin$ which
satisfies (C1)--(C5).

Let $\AAA$ be a full subcategory of $\CC$.
An object $F \in \Ob(\CC)$ is \emph{ultrahomogeneous for $\AAA$} if for every $A \in \Ob(\AAA)$ and every pair of morphisms
$e_1, e_2 \in \hom(A, F)$ there is a $g \in \Aut(F)$ such that $g \cdot e_1 = e_2$:
\begin{center}
  \begin{tikzcd}
    F \arrow[rr, "g"] & & F \\
     & A \arrow[ul, "e_1"] \arrow[ur, "e_2"'] &
  \end{tikzcd}
\end{center}
An object $F \in \Ob(\CC)$ is \emph{ultrahomogeneous in $\CC$ (with respect to $\CCfin$)} if it is
ultrahomogeneous for $\CCfin$.

Recall that an object $F$ is universal for $\AAA$ if $A \to F$ for all $A \in \Ob(\AAA)$.
We shall say that $F$ is \emph{universal in $\CC$ (with respect to $\CCfin$)} if it is universal for~$\CCfin$.
Let us define the \emph{age of $F$ in $\CC$ with respect to $\CCfin$} as
$$
  \Age_{(\CC,\CCfin)}(F) = \{A \in \Ob(\CCfin) : A \toCC F \}.
$$
Clearly, every $F$ is universal for its age.
Whenever $\CC$ and $\CCfin$ are fixed we shall simply write $\Age(F)$.

  Let $\DD$ be a full subcategory of a locally finite category $\CC$.
  An $F \in \Ob(\CC)$ is \emph{locally finite for $\DD$} if:
  \begin{itemize}
  \item
    for every $A, B \in \Ob(\DD)$ and every $e \in \hom(A, F)$, $f \in \hom(B, F)$ there exist $D \in \Ob(\DD)$,
    $r \in \hom(D, F)$, $p \in \hom(A, D)$ and $q \in \hom(B, D)$ such that $r \cdot p = e$ and $r \cdot q = f$:
    \begin{center}
        \begin{tikzcd}
          D \arrow[rr, "r"] & & F & &  \\
          & A \arrow[ul, "p"] \arrow[ur, "e"' near start] & & B \arrow[ulll, "q" near start] \ar[ul, "f"']
        \end{tikzcd}
    \end{center}
  \item
    and for every $H \in \Ob(\CC)$, $r' \in \hom(H, F)$, $p' \in \hom(A, H)$ and $q' \in \hom(B, H)$
    such that $r' \cdot p' = e$ and $r' \cdot q' = f$ there is an $s \in \hom(D, H)$ such that the diagram below commutes
    
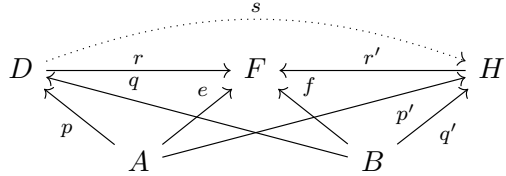
\begin{center}
        \begin{tikzcd}
          D \arrow[rr, "r"] \arrow[rrrr, dotted, bend left=20, "s"] & & F & & H \arrow[ll, "r'"'] \\
          & A \arrow[ul, "p"] \arrow[ur, "e" near end] \arrow[urrr, "p'"' near end] & & B \arrow[ulll, "q"' near end] \arrow[ul, "f"' near end]
              \arrow[ur, "q'"']
        \end{tikzcd}
    \end{center}
  \end{itemize}
  An $F \in \Ob(\CC)$ is \emph{locally finite in $\CC$ (with respect to $\CCfin$)} if it is locally finite for $\CCfin$.

  An \emph{expansion} of a category $\CC$ is a category $\CC^*$ together with a forgetful functor
  $U : \CC^* \to \CC$ (that is, a functor which is surjective on objects and injective on hom-sets).
  We shall generally follow the convention that $A, B, C, \ldots$ denote objects from $\CC$
  while $A^*, B^*, C^*, \ldots$ denote objects from $\CC^*$.
  Since $U$ is injective on hom-sets we may safely assume that
  $\hom_{\CC^*}(A^*, B^*) \subseteq \hom_\CC(A, B)$ where $A = U(A^*)$, $B = U(B^*)$.
  In particular, $\id_A^* = \id_A$ for $A = U(A^*)$. Moreover, it is safe to drop
  subscripts $\CC$ and $\CC^*$ in $\hom_\CC(A, B)$ and $\hom_{\CC^*}(A^*, B^*)$, so we shall
  simply write $\hom(A, B)$ and $\hom(A^*, B^*)$, respectively.
  Let
  $
    U^{-1}(A) = \{A^* \in \Ob(\CC^*) : U(A^*) = A \}
  $. Note that this is not necessarily a set. Therefore, we say that
  an expansion $U : \CC^* \to \CC$ is \emph{precompact} (cf.~\cite{vanthe-more}) if
  $U^{-1}(A)$ is a set for all $A \in \Ob(\CC)$, and it is a finite set for all $A \in \Ob(\CCfin)$. 

  An expansion $U : \CC^* \to \CC$ is \emph{reasonable} (cf.~\cite{KPT}) if
  for every $e \in \hom(A, B)$ and every $A^* \in U^{-1}(A)$ there is a $B^* \in U^{-1}(B)$ such that
  $e \in \hom(A^*, B^*)$:
  \begin{center}
      \begin{tikzcd}
        A^* \arrow[r, "e"] \arrow[d, dashed, mapsto, "U"'] & B^* \arrow[d, dashed, mapsto, "U"] \\
        A \arrow[r, "e"'] & B
      \end{tikzcd}
  \end{center}
  
  An expansion $U : \CC^* \to \CC$ has \emph{unique restrictions} if
  for every $B^* \in \Ob(\CC^*)$ and every $e \in \hom(A, U(B^*))$ there is a \emph{unique} $A^* \in U^{-1}(A)$
  such that $e \in \hom(A^*, B^*)$:
  \begin{center}
      \begin{tikzcd}
        \llap{\hbox{$\restr {B^*} e = \mathstrut$}}{A^*} \arrow[r, "e"] \arrow[d, dashed, mapsto, "U"'] & {B^*} \arrow[d, dashed, mapsto, "U"] \\
        A \arrow[r, "e"'] & B
      \end{tikzcd}
  \end{center}
  We denote this unique $A^*$ by $\restr {B^*} e$ and refer to it as the \emph{restriction of $B^*$ along~$e$}.

  Let $\AAA$ and $\AAA^*$ be categories and $U : \AAA^* \to \AAA$ and expansion.
  Following \cite{vanthe-more} we say that $U : \AAA^* \to \AAA$ \emph{has the expansion property}
  if for every $A \in \Ob(\AAA)$ there exists a $B \in \Ob(\AAA)$ such that $A^* \to B^*$ for all
  $A^*, B^* \in \Ob(\AAA^*)$ with $U(A^*) = A$ and $U(B^*) = B$.
  If $\AAA^*$ is directed and all the morphisms in $\AAA$ are mono, and if
  $U : \AAA^* \to \AAA$ is a reasonable expansion with unique restrictions such that $U^{-1}(A)$ is finite
  for all $A \in \Ob(\AAA)$ then $U : \AAA^* \to \AAA$ has the expansion property if and only if
  for every $D^* \in \Ob(\AAA^*)$ there is a $B \in \Ob(\AAA)$
  such that for all $B^* \in \Ob(\AAA^*)$ with $U(B^*) = B$ we have $D^* \to B^*$~\cite{masul-kpt}.

  Let $F$ be a locally finite object. For $A \in \Ob(\CCfin)$ such that $A \to F$ and $e_1, e_2 \in \hom(A, F)$ let
  $
    N_F(e_1, e_2) = \{ f \in \Aut(F) : f \cdot e_1 = e_2 \}
  $.
  Then
  $$
    M_F = \{ N_F(e_1, e_2) : A \in \Ob(\CCfin), \; A \toCC F \text{ and } e_1, e_2 \in \hom(A, F) \}.
  $$
  is a base of a topology $\tau_F$ on $\Aut(F)$~\cite{masul-kpt}.
  
  Let $U : \CC^* \to \CC$ be a reasonable precompact expansion with unique restrictions and let $F \in \Ob(\CC)$.
  For $A \in \Ob(\CCfin)$,  $e \in \hom(A, F)$ and $A^* \in U^{-1}(A)$ let
  $$
    N(e, A^*) = \{ F^* \in U^{-1}(F) : e \in \hom(A^*, F^*) \}.
  $$
  Then $\calS_F = \{ N(e, A^*) : U(A^*) \in \Ob(\CCfin), e \in \hom(U(A^*), F) \}$
  is a base of clopen sets of a topology $\sigma_F$ on $U^{-1}(F)$~\cite{masul-kpt}.
  
  Each reasonable expansion with unique restrictions yields
  an action of $\Aut(F)$ on $U^{-1}(F)$ for every $F \in \Ob(\CC)$:
  for $F \in \Ob(\CC)$, $g \in \Aut(F)$ and $F^* \in U^{-1}(F)$ let $F^* \cdot g$ denote the unique
  element of $U^{-1}(F)$ satisfying $g \in \hom(F^*\cdot g, F^*)$. (See~\cite{masul-kpt} for details.)
  This action is continuous with respect to topologies $\tau_F$ on $\Aut(F)$ and $\sigma_F$ on $U^{-1}(F)$ and will be referred to as \emph{logical}.

That every \Fraisse\ class with finite Ramsey degrees has a Ramsey expansion
was first established in~2016 in~\cite{zucker1}. A combinatorial proof of the same fact was then given in~\cite{vanthe-finramdeg},
and this was put into the context of category theory in~\cite{masul-kpt}. The categorical version takes the following form:

\begin{THM} \cite{zucker1,vanthe-finramdeg,masul-kpt}\label{thm-zucker}
  Let $\CC$ be a locally small category and let $\CCfin$ be a full subcategory of $\CC$ such that (C1) -- (C5) hold.
  Let $F \in \Ob(\CC)$ be an ultrahomogeneous locally finite object, and let $\AAA$ be the full subcategory of
  $\CCfin$ spanned by $\Age(F)$. Then the following are equivalent:
  
  (1) $\AAA$ has finite Ramsey degrees.
  
  (2) There is a reasonable precompact expansion with unique restrictions $U : \CC^* \to \CC$ and
  a full subcategory $\AAA^*$ of $\CC^*$ which is directed, has the Ramsey property and
  $\restr{U}{\AAA^*} : \AAA^* \to \AAA$ has the expansion property.
\end{THM}

\section{Weak amalgamation and weak \Fraisse\ categories}
\label{srd-wap.sec.weak-fraisse-cats}

Let $\CC$ be a locally small category whose morphisms are mono.
We say that $\CC$ has the \emph{weak amalgamation property}~\cite{kubis-weak-fraisse-cat}
if for any $A \in \Ob(\CC)$ there is $A' \in \Ob(\CC)$ and a morphism $f \in \hom(A, A')$
so that whenever we are given $B, C \in \Ob(\CC)$ and morphisms $g \in \hom(A', B)$ and $h \in \hom(A', C)$
there are $D \in \Ob(\CC)$ and morphisms $r \in \hom(B, D)$ and $s \in \hom(C, D)$ so that
$r \cdot g \cdot f = s \cdot h \cdot f$:
\begin{center}
    \begin{tikzcd}[row sep=tiny]
              & A' \arrow[r, "g"] & B \arrow[dr, "r"] & \\
        A \arrow[ur, "f"] \arrow[dr, "f"'] &  &  &  D\\
              & A' \arrow[r, "h"'] & C \arrow[ur, "s"'] & 
    \end{tikzcd}
\end{center}
When $f$ is as above we call such an $f$ an \emph{amalgamation arrow for A}.

\begin{THM}\label{srd-wap.thm.srd=>wap}
  Let $\CC$ be a locally small directed category with finite Ramsey degrees.
  Then $\CC$ has the weak amalgamation property.
\end{THM}
\begin{proof}
  The following weakening of the amalgamation property will be useful during the course of the proof.
  Let $\CC$ be a locally small category whose morphisms are mono and fix $k < \omega$.
  We say that $A \in \Ob(\CC)$ has \emph{2-out-of-$k$-amalgamation}
  if for any $B_0,\ldots,B_{k-1} \in \Ob(\CC)$ and morphisms $g_i \in \hom(A, B_i)$ there
  are $C \in \Ob(\CC)$, $i \ne j < k$ and morphisms $r \in \hom(B_i, C)$ and $s \in \hom(B_j, C)$
  with $r \cdot g_i = s \cdot g_j$:
  \begin{center}
      \begin{tikzcd}[column sep=small]
              &        &     &   C    &     &        &         \\
          B_0 & \cdots & B_i \arrow[ur, "r"] & \cdots & B_j \arrow[ul, "s"'] & \cdots & B_{k-1} \\
              &        &     &
              A \arrow[ulll, "g_0"] \arrow[ul, "g_i"'] \arrow[ur, "g_j"] \arrow[urrr, "g_{k-1}"']
              &     &        &
      \end{tikzcd}
  \end{center}

  The proof now proceeds in two steps.

  \bigskip

  Claim 1. Assume that $t_\CC(A) = k - 1$ for some $A \in \Ob(\CC)$. Then $A$ has 2-out-of-$k$-amalgamation.

  Proof.
  Fix morphisms $f_i \in \hom(A, B_i)$ with $B_i \in \Ob(\CC)$ for each $i < k$. Then use the
  fact that $\CC$ is directed to find some $C \in \Ob(\CC)$ and morphisms $g_i \in \hom(B_i, C)$ for each $i < k$.
  Next find $D \in \Ob(\CC)$ satisfying
  $$
    D \longrightarrow (C)^A_{k,k-1}.
  $$
  Consider the coloring $\chi\colon \hom(A, D) \to k$ where $\chi(h) = i < k - 1$ if $i$ is least so that there is
  $g \in \hom(B_i, D)$ with $h = g \cdot f_i$, or set $\chi(h) = k - 1$ if there is no such $i < k - 1$.
  Then there is an $x \in \hom(C, D)$ such that $|\chi(x \cdot \hom(A, C))| \le k - 1$.
  In particular, there is some color $j < k$ which is avoided, that is, $j \notin \chi(x \cdot \hom(A, C))$.
  Then consider the value of $\chi(x \cdot g_j \cdot f_j) = i \ne j$. So there
  is $g \in \hom(B_i ,D)$ with $g \cdot f_i = x \cdot g_j \cdot f_j$, showing that $f_i$ and $f_j$ can be amalgamated.
  This concludes the proof of Claim~1.

  \bigskip

  Claim 2. Suppose for every $A \in \Ob(\CC)$ there is some $k \in \NN$ so that $A$ has 2-out-of-$k$-amalgamation.
  Then $\CC$ has the weak amalgamation property.

  Proof.
  Suppose $\CC$ failed to have the weak amalgamation property as witnessed by $A \in \Ob(\CC)$. Then the identity map
  $\id_A$ is not an amalgamation arrow. Set $C_{0} = A$ and $\id_A = f_0 \in \hom(A, C_{0})$.
  Now suppose $f_i \in \hom(A, C_{i})$ has been defined. Then $f_i$ is not an amalgamation arrow, so we may
  find $B_{i+1}, C_{i+1} \in \Ob(\CC)$ and morphisms $g_i \in \hom(C_{i}, B_{i+1})$ and $h_i \in \hom(C_{i}, C_{i+1})$
  so that $g_i \cdot f_i$ and $h_i \cdot f_i$ cannot be amalgamated.
  Then set $f_{i+1} = h_i \cdot f_i$. We can continue this as long as we would
  like, obtaining arrows $\{g_i \cdot f_i : i < \omega\}$ no pair of which can be amalgamated.
  Therefore, $A$ does not have 2-out-of-$k$-amalgamation.
  This concludes the proof of Claim~2 and proof of the theorem.
\end{proof}

The remainder of this section is devoted to the exposition of key notions of weak \Fraisse\ theory~\cite{kubis-weak-fraisse-cat}
which is one of the building blocks of our construction.

A subcategory $\DD$ of $\CC$ is \emph{weakly dominating in $\CC$}~\cite{kubis-weak-fraisse-cat}
if it is cofinal in $\CC$
and for every $D \in \Ob(\DD)$ there exist a $D' \in \Ob(\DD)$ and a $j \in \hom_\DD(D, D')$
such that for every $C \in \Ob(\CC)$ and every $f \in \hom_\CC(D', C)$
there is an $E \in \Ob(\DD)$ and a morphism $g \in \hom_\CC(C, E)$ such that
$g \cdot f \cdot j$ is a morphism in $\DD$.

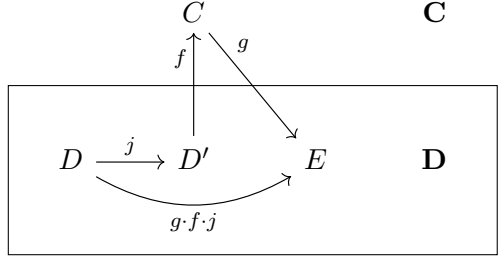
\begin{center}
  \begin{tikzcd}[execute at end picture={
    \draw (-3.25,-2.25) rectangle (3.25,0.0);
  }]
                     & C \arrow[ddr, "g" near start]  & & \CC\\
                     &  & & \\
    D \arrow[r, "j"] \arrow[rr, bend right, "g \cdot f \cdot j"'] & D' \arrow[uu, "f" near end] & E & \DD
  \end{tikzcd}
\end{center}

A category $\CC$ is a \emph{weak \Fraisse\ category}~\cite{kubis-weak-fraisse-cat}
if it is directed, has the weak amalgamation property and is weakly dominated by a countable subcategory.

A sequence $W = (W_n, w_n^m)_{n \le m \in \omega}$ is a \emph{weak \Fraisse\ sequence}~\cite{kubis-weak-fraisse-cat}
if the following is satisfied:
\begin{itemize}
\item for every $C \in \Ob(\CC)$ there is an $n \in \omega$ such that $C \toCC W_n$; and
\item for every $n \in \omega$ there exists an $m \ge n$ such that for every
      $f \in \hom_\CC(W_m, C)$ there are $k \ge m$ and $g \in \hom_\CC(C, W_k)$
      satisfying $g \cdot f \cdot w^m_n = w^k_n$.
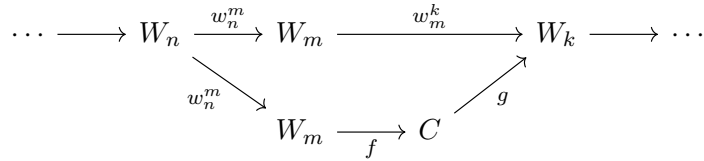
      \begin{center}
        \begin{tikzcd}
          \cdots \arrow[r] & W_n \arrow[r, "w_n^m"] \arrow[dr, "w_n^m"'] & W_m \arrow[rr, "w_m^k"]  & & W_k \arrow[r] & \cdots \\
                           &                        &        W_m \arrow[r, "f"']   & C \arrow[ur, "g"'] & &
        \end{tikzcd}
      \end{center}
\end{itemize}

If follows immediately that every category with a weak \Fraisse\ sequence is directed and has the weak amalgamation property~\cite{kubis-weak-fraisse-cat}.

A particularly important class of weak \Fraisse\ categories is the category of chains formed from
a category with a weak \Fraisse\ limit~\cite{kubis-weak-fraisse-cat, kubis-fraisse-seq, bbbk-weak-ramsey}.
Let $\omega = \{0, 1, 2, \ldots\}$ denote the chain of nonnegative integers treated here
as a poset category (that is, for all $n, m \in \omega$ such that $n \le m$ there is a single morphism $n \to m$).
A \emph{sequence in a category $\CC$} is any functor $X : \omega \to \CC$. We shall find it more convenient
to describe functors $X : \omega \to \CC$ as $(X_n, x_n^m)_{n \le m \in \omega}$ where $X_n = X(n) \in \Ob(\CC)$
and $x_n^m \in \hom_\CC(X_n, X_m)$ is the image under $X$ of the only morphism $n \to m$.
Then, clearly, $x_n^n = \id_{X_n}$ and $x_m^k \cdot x_n^m = x_n^k$ whenever $n \le m \le k$.

The next step is to define morphisms between sequences in such a way that
a morphism from a sequence $X$ to a sequence $Y$ induces a morphism from the colimit of $X$ into the
colimit of $Y$ whenever the category of sequences is embedded into a category in which sequences
$X$ and $Y$ have colimits.
This is a two-phase process: the idea is to start with the category of sequences and transformations $\sigma_1\CC$
and then factor this category by a congruence to arrive at the category $\sigma_0\CC$
of sequences and morphisms between them. Hence, a morphism of sequences is an equivalence class of transformations between them.

Let $(X_n, x_n^m)_{n \le m \in \omega}$ and $(Y_n, y_n^m)_{n \le m \in \omega}$ be sequences in $\CC$.
A \emph{transformation} $(X_n, x_n^m)_{n \le m \in \omega} \to (Y_n, y_n^m)_{n \le m \in \omega}$
is a pair $(F, \phi)$ where $\phi : \omega \to \omega$ is a functor such that $\phi(\omega)$ is cofinal in $\omega$
and $F : X \to Y \circ \phi$ is a natural transformation. In other words, $\phi$ is a nondecreasing cofinal map $\omega \to \omega$
(that is, $i \le j \Rightarrow \phi(i) \le \phi(j)$ and for every $n$ there is an $m$ such that $n \le \phi(m)$)
and there is a family of arrows $F_n : X_n \to Y_{\phi(n)}$, $n \in \omega$, such that
\begin{center}
  \begin{tikzcd}
    X_n \arrow[d, "F_n"'] \arrow[rr, "x_n^m"] & & X_m \arrow[d, "F_m"]\\
    Y_{\phi(n)} \arrow[rr, "y_{\phi(n)}^{\phi(m)}"'] & & Y_{\phi(m)}
  \end{tikzcd}
  \quad
  \text{for all $n \le m \in \omega$.}
\end{center}
All sequences in $\CC$ and transformations between them form a category that we denote by $\sigma_1\CC$.

Two transformations $(F, \phi), (G, \psi) : (X_n, x_n^m)_{n \le m \in \omega} \to (Y_n, y_n^m)_{n \le m \in \omega}$
are \emph{equivalent}, in symbols $(F, \phi) \approx (G, \psi)$,
if for every $n \in \omega$ there exists an $m \ge \max\{\phi(n), \psi(n)\}$ such that
\begin{center}
  \begin{tikzcd}
     X_n \arrow[d, "F_n"'] \arrow[dr, "G_n"] \\
     Y_{\phi(n)} \arrow[rr, bend right, "y_{\phi(n)}^m"'] & Y_{\psi(n)} \arrow[r, "y_{\psi(n)}^m"] & Y_m
  \end{tikzcd}
\end{center}
It is easy to check that $\approx$ is a congruence of $\sigma_1\CC$, so let $\sigma_0\CC = \sigma_1\CC / \Boxed{\approx}$
be the factor category. Clearly, each morphism is an equivalence class of transformations.
Just as a quick demonstration of the interaction of all these notions let us show the following

\begin{LEM}\label{srd-wap.lem.monos}
  Let $\CC$ be a locally small category whose morphisms are mono. Then all the morphisms in $\sigma_0\CC$ are mono.
\end{LEM}
\begin{proof}
  Let $X = (X_n, x_n^m)_{n \le m \in \omega}$, $Y = (Y_n, y_n^m)_{n \le m \in \omega}$ and
  $Z = (Z_n, z_n^m)_{n \le m \in \omega}$ be sequences in $\CC$, and let $(F, \phi) : Y \to Z$
  and $(G, \psi), (H, \theta) : X \to Y$ be transformations such that $(F, \phi) \cdot (G, \psi) \approx (F, \phi) \cdot (H, \theta)$.
  We are going to show that $(G, \psi) \approx (H, \theta)$.

  Note that $(F, \phi) \cdot (G, \psi) \approx (F, \phi) \cdot (H, \theta)$
  means that for every $n \in \omega$ there is an $m \ge \max\{\phi(\psi(n)), \phi(\theta(n))\}$ such that
  \begin{equation}\label{crt.eq.weak-amalg-mono-1}
      z_{\phi(\psi(n))}^{m} \cdot F_{\psi(n)} \cdot G_n = z_{\phi(\theta(n))}^{m} \cdot F_{\theta(n)} \cdot H_n
  \end{equation}
  \begin{center}
    \begin{tikzcd}
       X_n \arrow[d, "G_n"'] \arrow[dr, "H_n"] \\
       Y_{\psi(n)} \arrow[d, "F_{\psi(n)}"'] & Y_{\theta(n)} \arrow[d, "F_{\theta(n)}"] \\
       Z_{\phi(\psi(n))} \arrow[rr, bend right] & Z_{\phi(\theta(n))} \arrow[r] & Z_m
    \end{tikzcd}
  \end{center}
  Since $\phi$, $\psi$ and $\theta$ are nondecreasing and cofinal, $\phi \circ \theta$ and
  $\phi \circ \psi$ are also nondecreasing and cofinal, so there is a $k \in \omega$ such that $\theta(k) > \theta(n)$ and
  $\phi(\theta(k)) > m$, and an $s \in \omega$ such that $\psi(s) > \theta(k)$ and $\phi(\psi(s)) > \phi(\theta(k))$:
  \begin{center}
    \begin{tikzcd}
       X_n \arrow[d, "G_n"'] \arrow[dr, "H_n" near start] \\
       Y_{\psi(n)} \arrow[rrrr, bend left=15] \arrow[d, "F_{\psi(n)}"'] & Y_{\theta(n)} \arrow[d, "F_{\theta(n)}"] \arrow[rr] & & Y_{\theta(k)} \arrow[d, "F_{\theta(k)}"] \arrow[r] & Y_{\psi(s)} \arrow[d, "F_{\psi(s)}"] \\
       Z_{\phi(\psi(n))} \arrow[rrrr, bend right=15] \arrow[rr, bend right=12] & Z_{\phi(\theta(n))} \arrow[r] & Z_m \arrow[r] & Z_{\phi(\theta(k))} \arrow[r] & Z_{\phi(\psi(s))}
    \end{tikzcd}
  \end{center}
  The rest of the proof reduces to straightforward calculation. Multiplying \eqref{crt.eq.weak-amalg-mono-1}
  by $z_{m}^{\phi(\psi(s))}$ from the left we get
  $$
    z_{m}^{\phi(\psi(s))} \cdot z_{\phi(\psi(n))}^{m} \cdot F_{\psi(n)} \cdot G_n = z_{m}^{\phi(\psi(s))} \cdot z_{\phi(\theta(n))}^{m} \cdot F_{\theta(n)} \cdot H_n
  $$
  that is
  $$
    z_{\phi(\psi(n))}^{\phi(\psi(s))} \cdot F_{\psi(n)} \cdot G_n = z_{\phi(\theta(n))}^{\phi(\psi(s))} \cdot F_{\theta(n)} \cdot H_n.
  $$
  Using the fact that $F : Y \to Z$ is a transformation as the next step we have
  $$
    F_{\psi(s)} \cdot y_{\psi(n)}^{\psi(s)} \cdot G_n = F_{\psi(s)} \cdot y_{\theta(n)}^{\psi(s)} \cdot H_n.
  $$
  Finally, $F_{\psi(s)}$ is mono as a morphism in $\CC$, whence
  $$
    y_{\psi(n)}^{\psi(s)} \cdot G_n = y_{\theta(n)}^{\psi(s)} \cdot H_n.
  $$
  This shows that $(G, \psi) \approx (H, \theta)$.
\end{proof}

The category $\CC$ embeds fully into $\sigma_0\CC$ as follows.
For $A \in \Ob(\CC)$ let $\overline A = (A, \id_A)_{n \le m \in \omega}$ denote the constant sequence such that
$A_n = A$ for all $n \in \omega$ and $a_n^m = \id_A$ for all $n \le m \in \omega$.
Every morphism $f \in \hom_\CC(A, B)$ gives rise to a unique transformation
$\overline f = (\Const_f, \id_\omega)$ where $(\Const_f)_n = f$, $n \in \omega$.
It is easy to check that $J : \CC \to \sigma_0\CC : A \mapsto \overline A : f \mapsto \overline f / \Boxed{\approx}$
is indeed a full functor injective on objects, and hence an embedding~\cite{kubis-fraisse-seq}.

It is a bit technical but easy to show (see~\cite{kubis-fraisse-seq}) that $\sigma_0\CC$ is a cocompletion of $\CC$:
for every sequence $X = (X_n, x_n^m)_{n \le m \in \omega}$ in $\CC$ we have simply added a formal colimit
to $\sigma_0\CC$ and adjusted the morphisms so that $X$ is the colimit of $J(X)$ in $\sigma_0\CC$. More
precisely, the following holds in $\sigma_0\CC$ (see~\cite{kubis-fraisse-seq} for details):
$$
  X = \operatorname{colim}(\overline{X_0} \overset{\overline x_0^1}\longrightarrow \overline{X_1} \overset{\overline x_1^2 }\longrightarrow \cdots)
$$
Clearly, every element in $\Ob(\sigma_0\CC)$ is a colimit of a sequence in~$\CC$.

It comes as no surprise that weak \Fraisse\ sequences should demonstrate a certain level of homogeneity
in $\sigma_0\CC$ with respect to objects from $J(\CC)$.
In the setting of weak \Fraisse\ theory the corresponding notion is referred to as weak homogeneity.

Let $\CC$ be a locally small category whose morphisms are mono, let $\DD$ be a full subcategory of $\CC$ and let
$S \in \Ob(\CC)$. We say that $S$ is \emph{weakly homogeneous for $\DD$} \cite{kubis-weak-fraisse-cat}
if for every $A \in \Ob(\DD)$ and every $f \in \hom(A, S)$ there is a $B \in \Ob(\DD)$, an $e \in \hom(A, B)$
and an $i \in \hom(B, S)$ such that
\begin{description}
  \item[(W1)] $f = i \cdot e$, and
  \item[(W2)] for every $j \in \hom(B, S)$ there is an $h \in \Aut(S)$ such that $i \cdot e = h \cdot j \cdot e$:
        \begin{center}
            \begin{tikzcd}[row sep=tiny]
                   & B \arrow[r, "i"] & S \\
             A \arrow[ur, "e"] \arrow[dr, "e"'] \arrow[rru, bend left=45, "f"] \\
                   & B \arrow[r, "j"'] & S \arrow[uu, "h"']
            \end{tikzcd}
      \end{center}
\end{description}

\begin{THM}\cite{kubis-weak-fraisse-cat}\label{srd-wap.thm.kubis-weak-fraisse}
  Let $\CC$ be a category.
  
  $(a)$ $\CC$ is a weak \Fraisse\ category if and only if there is a weak \Fraisse\ sequence in~$\CC$.

  $(b)$ A category may have, up to isomorphism, at most one weak \Fraisse\ sequence.

  $(c)$ If $\CC$ is a weak \Fraisse\ category and $W$ a weak \Fraisse\ sequence in $\CC$ then
  $W$ as an object of $\sigma_0\CC$ is weakly homogeneous for $\CC$.
\end{THM}

In case $\CC$ is a class of finite first-order structures of the same first-order signature $\Theta$, we can think of $\CC$ as
a category where embeddings serve as morphisms, and in this particular case we can take $\sigma_0\CC$ to
be the class of all structures isomorphic to the unions of countable chains in~$\CC$.
Recall that the morphisms between sequences in $\sigma_0\CC$ are defined in a rather convoluted manner
to ensure that in the context of first-order structures a morphism from a sequence $X$ to a sequence $Y$
correspond uniquely to an embedding of the colimit of~$X$ into the colimit of~$Y$.

\section{Weak homogeneity and precompact expansions}
\label{srd-wap.sec.weak-hom-precomp-exp}

The purpose of this section is to prove a generalization of Theorem~\ref{thm-zucker}
where ultrahomogeneity is replaced by weak homogeneity.
Interestingly, the proof as we have presented it in~\cite{masul-kpt} remains largely the same, so
we cover only the differences here. We strongly suggest the reader to have a copy of the proof given in~\cite{masul-kpt}
at hand while reading this section.


\begin{LEM}\label{crt.lem.EP-1-weak}
  Let $U : \CC^* \to \CC$ be a reasonable expansion with unique restrictions.
  Let $F$ be a locally finite object of $\CC$ which is weakly homogeneous for its age
  and let $G = \Aut(F)$. Then for $F^*, F_1^* \in U^{-1}(F)$ we have
  $F_1^* \in \overline{F^*\cdot G}$ (where the closure is computed in $\sigma_F$) if and only if
  $\Age(F_1^*) \subseteq \Age(F^*)$.
\end{LEM}
\begin{proof}
  $(\Rightarrow)$
  The same as the proof of direction $(\Rightarrow)$ in \cite[Lemma~5.11]{masul-kpt}.
  
  $(\Leftarrow)$
  Assume that $\Age(F_1^*) \subseteq \Age(F^*)$ and let us show that every neighborhood of $F_1^*$
  intersects $F^* \cdot G$. Let $N(e, A^*)$ be a neighborhood of $F_1^*$.
  Then $e \in \hom_{\CC^*}(A^*, F_1^*)$ and $e \in \hom_{\CC}(A, F)$. 
  Our intention now is to show that $e \in \hom_{\CC^*}(A^*, F^* \cdot g)$ for some $g \in G$.
  
  Since $F$ is weakly homogeneous
  there exists a $B \in \Ob(\CCfin)$ and morphisms $h \in \hom_\CC(A, B)$ and $i \in \hom_\CC(B, F)$
  such that $e = i \cdot h$. Let $B_1^* = \restr{F_1^*}{i}$. Then $B_1^* \in \Age(F_1^*) \subseteq \Age(F^*)$, so
  there exists a morphism $f \in \hom(B_1^*, F^*)$. Since $F$ is weakly homogeneous, there is a $g \in \Aut(F)$
  such that $g \cdot f \cdot h = i \cdot h = e$:
  \begin{center}
    \begin{tikzcd}[row sep=tiny]
           & B \arrow[r, "i"] & F \\
     A \arrow[ur, "h"] \arrow[dr, "h"'] \arrow[rru, bend left=45, "e"] \\
           & B \arrow[r, "f"'] & F \arrow[uu, "g"']
    \end{tikzcd}
\end{center}
  But then (see \cite[Lemma 5.1~(c)]{masul-kpt}):
\begin{center}
    \begin{tikzcd}
        A^* \arrow[r, "h"'] \arrow[rrr, bend left=20, "e"] \arrow[d, dashed, mapsto, "U"']
      & B_1^* \arrow[r, "f"'] \arrow[d, dashed, mapsto, "U"']
      & F^* \arrow[r, "g"'] \arrow[d, dashed, mapsto, "U"']
      & F^* \cdot g^{-1} \arrow[d, dashed, mapsto, "U"]
    \\
      A \arrow[r, "h"] \arrow[rrr, bend right=20, "e"'] & B \arrow[r, "f"] & F \arrow[r, "g"] & F
    \end{tikzcd}
\end{center}
  whence follows that $F^* \cdot g^{-1} \in N(e, A^*)$.
\end{proof}

\begin{PROP}\label{crt.prop.EP-2-weak}
  Let $U : \CC^* \to \CC$ be a reasonable precompact expansion with unique restrictions.
  Let $F$ be a locally finite weakly homogeneous object in $\CC$ and assume that $U^{-1}(F)$ is compact
  with respect to the topology $\sigma_F$.
  Let $G = \Aut(F)$ and let $F^* \in U^{-1}(F)$ be arbitrary. Then
  $\restr{U}{\Age(F^*)} : \Age(F^*) \to \Age(F)$ has the expansion property if and only if
  $\Age(F^*) = \Age(F_1^*)$ for all $F_1^* \in \overline{F^* \cdot G}$.
\end{PROP}
\begin{proof}
  $(\Rightarrow)$
  The same as the proof of direction $(\Rightarrow)$ in~\cite[Lemma~5.12]{masul-kpt}. 

  $(\Leftarrow)$
  Assume that $\Age(F^*) \subseteq \Age(F_1^*)$ for all $F_1^* \in \overline{F^* \cdot G}$. Let
  $A^* \in \Age(F^*)$ be arbitrary and let $A = U(A^*)$. For $e \in \hom(A, F)$ let
  $$
    X_e = \overline{F^* \cdot G} \sec N(e, A^*).
  $$
  Let us show that
  $$
    \overline{F^* \cdot G} = \UNION \{ X_e : e \in \hom(A, F) \}.
  $$
  The inclusion $\supseteq$ is trivial, while the inclusion $\subseteq$ follows from the assumption. Namely,
  if $F_1^* \in \overline{F^* \cdot G}$ then $\Age(F^*) \subseteq \Age(F_1^*)$; so
  $A^* \in \Age(F_1^*)$, or, equivalently, there is a morphism
  $f \in \hom(A^*, F_1^*)$, whence $F_1^* \in X_f$.

  By construction each $X_e$ is open in $\overline{F^* \cdot G}$.
  Since $\overline{F^* \cdot G}$ is compact (as a closed subspace of the compact space $U^{-1}(F)$), there is a finite sequence
  $e_0, \ldots, e_{k-1} \in \hom(A, F)$ such that
  $$
    \overline{F^* \cdot G} = \UNION \{ X_{e_j} : j < k \}.
  $$

  Since $F$ is locally finite, there exist $B \in \Ob(\CCfin)$ and morphisms $r \in \hom(B, F)$
  and $p_i \in \hom(A, B)$ such that $r \cdot p_i = e_i$, $i < k$. Moreover, $F$ is weakly homogeneous
  so for $r \in \hom(B, F)$ there exist $C \in \Ob(\CCfin)$ and
  morphisms $h \in \hom(B, C)$ and $i \in \hom(C, F)$ such that (W1), that is $i \cdot h = r$, and (W2) are satisfied.
  Let us show that for every $C^* \in U^{-1}(C)$ we have that $A^* \toCCC C^*$.

  Take any $C^* \in \Age(F^*)$ such that $U(C^*) = C$ and let $s \in \hom(C^*, F^*)$ be any morphism.
  Then, $s \in \hom(C, F)$, so by (W2) there is a $g \in G$ such that $g \cdot s \cdot h = i \cdot h = r$.
  Furthermore, let $B^* = \restr{C^*}{h}$. Since $g \in \hom(F^*, F^* \cdot g^{-1})$ we have that
\begin{center}
    \begin{tikzcd}
        F^* \cdot g^{-1} \arrow[d, dashed, mapsto, "U"']
      & B^*          \arrow[d, dashed, mapsto, "U"'] \arrow[l, "g \cdot s \cdot h"] \arrow[r, "h"]
      & C^*          \arrow[d, dashed, mapsto, "U"'] \arrow[r, "s"']
      & F^*          \arrow[d, dashed, mapsto, "U"]  \arrow[lll, bend right=20, "g"']
    \\
      F & B \arrow[r, "h"] \arrow[l, "r"'] & C \arrow[r, "s"] & F \arrow[lll, bend left=20, "g"] 
    \end{tikzcd}
\end{center}
  In particular, $r = g \cdot s \cdot h \in \hom(B^*, F^* \cdot g^{-1})$, so $B^* \in \Age(F^* \cdot g^{-1})$.
  Now, $F^* \cdot g^{-1} \in \overline{F^* \cdot G} = \UNION \{ X_{e_j} : j < k \}$,
  so $F^* \cdot g^{-1} \in X_{e_i}$ for some~$i$. Moreover, $r \cdot p_i = e_i$ by the construction of $B$.
  Therefore:
\begin{center}
    \begin{tikzcd}
        A^*          \arrow[d, dashed, mapsto, "U"'] \arrow[r, "e_i"] 
      & F^* \cdot g^{-1} \arrow[d, dashed, mapsto, "U"'] 
      & B^*          \arrow[d, dashed, mapsto, "U"]  \arrow[l, "r"']
    \\
      A \arrow[rr, bend right=20, "p_i"'] \arrow[r, "e_i"] & F & B \arrow[l, "r"']
    \end{tikzcd}
\end{center}
  Let $A_1^* = \restr{B^*}{p_i}$. Since $B^* = \restr{F^* \cdot g^{-1}}{r}$ we have
  $A_1^* = \restr{(\restr{F^* \cdot g^{-1}}{r})}{p_i} = \restr{F^* \cdot g^{-1}}{r \cdot p_i} =
  \restr{F^* \cdot g^{-1}}{e_i} = A^*$. Consequently, $p_i \in \hom(A^*, B^*)$ which, together with
  $h \in \hom(B^*, C^*)$ concludes the proof that $A^* \toCCC C^*$.
\end{proof}

Putting it all together and having in mind parts of the proof from~\cite{masul-kpt} that do not depend on $F$ being homogeneous
we finally get the following:

\begin{THM}\label{crt.thm.akpt-main2}
  Let $\CC$ be a locally small category and let $\CCfin$ be a full subcategory of $\CC$ such that (C1) -- (C5) hold.
  Let $F \in \Ob(\CC)$ be a weakly homogeneous locally finite object, and let $\AAA$ be the full subcategory of
  $\CCfin$ spanned by $\Age(F)$. Then the following are equivalent:
  
  (1) $\AAA$ has finite Ramsey degrees.
  
  (2) There is a reasonable precompact expansion with unique restrictions $U : \CC^* \to \CC$ and
  a full subcategory $\AAA^*$ of $\CC^*$ which is directed, has the Ramsey property and
  $\restr{U}{\AAA^*} : \AAA^* \to \AAA$ has the expansion property.
\end{THM}

      For reader's convenience we conclude the section with the quick recapitulation of the construction of $\CC^*$ and $\AAA^*$.
      The category $\CC^*$ is constructed from $\CC$ by adding structure to objects of $\CC$. The language convenient for
      the efficient description of the additional structure is that of essential colorings which capture the small Ramsey degree of
      an object by identifying the ``unavoidable coloring'' (see~\cite{zucker1} and also~\cite{vanthe-finramdeg}).
      Given a locally finite $F \in \Ob(\CC)$, a coloring $\lambda : \hom(A, B) \to t$, $t \ge 2$, is \emph{essential at $B$} if
      for every coloring $\chi : \hom(A, F) \to k$ there is a $w \in \hom(B, F)$
      such that $\ker \lambda \subseteq \ker \chi^{(w)}$, where $\chi^{(w)}(f) = \chi(w \cdot f)$.
      A coloring $\gamma : \hom(A, F) \to t$, $t \ge 2$, is \emph{essential} if
      for every $B \in \Ob(\AAA)$ such that $A \toCC B$ and every $w \in \hom(B, F)$
      the coloring $\gamma^{(w)} : \hom(A, B) \to t$ is essential at~$B$.
      The key observation now is that for every $A \in \Ob(\AAA)$ there is an essential coloring
      $\gamma_A : \hom(A, F) \to t_\AAA(A)$~(see~\cite{zucker1,vanthe-finramdeg} for the original statement
      and \cite{masul-kpt} for the proof of the result in the categorical setting).

      Let $\CC^*$ be the category whose objects are pairs $C_\theta = (C, \theta)$ where
      $C \in \Ob(\CC)$ and $\theta = (\theta_A)_{A \in \Ob(\AAA)}$ is a family of colorings
      $$
          \theta_A : \hom(A, C) \to t_\AAA(A)
      $$
      indexed by the objects of $\AAA$. Morphisms in $\CC^*$ are morphisms from $\CC$ that preserve colorings. More precisely,
      $f$ is a morphism from $C_\theta = (C, \theta)$ to $D_\delta = (D, \delta)$ in $\CC^*$
      if $f \in \hom(C, D)$ and
      $$
          \delta(f \cdot e) = \theta(e), \text{ for all } e \in \UNION_{A \in \Ob(\AAA)} \hom(A, C).
      $$
      Let $U : \CC^* \to \CC$ be the obvious forgetful functor $(C, \theta) \mapsto C$ and $f \mapsto f$.
      This is a reasonable precompact expansion with unique restrictions~\cite{masul-kpt}.

      Let $\gamma = (\gamma_A)_{A \in \Ob(\AAA)}$, and let $F_\gamma = (F, \gamma) \in \Ob(\CC^*)$ be an arbitrary
      expansion of $F$ (the weakly homogeneous locally finite object from the proof). As we have seen, expansions with unique restrictions
      induce group actions. Moreover, the action of $G = \Aut(F)$ on $U^{-1}(F)$ is continuous~\cite{masul-kpt}.
      Since $U^{-1}(F)$ is compact \cite{masul-kpt} there is an $F^* = (F, \phi^*) \in \overline{F_\gamma \cdot G}$
      such that $\overline{F^* \cdot G}$ is minimal with respect to inclusion. We then let
      $\AAA^* = \Age(F^*)$.

\section{The Main Result}
\label{srd-wap.sec.main}

We have now set up all the infrastructure necessary for the main result of the paper.
We shall say that $\CC$ is a \emph{category of finite objects} if $\CC$ has a skeleton $\SS$
with the following properties:
\begin{itemize}
  \item $\SS$ is a countable category,
  \item $\hom_\SS(A, B)$ is finite for all $A, B \in \SS$,
  \item for every $B \in \Ob(\SS)$ the set $\{A \in \Ob(\SS) : A \toSS B\}$ is finite, and
  \item every $A \in \Ob(\SS)$ is locally finite for $\SS$.
\end{itemize}

Whenever $\CC$ is a category of finite structures, we take $\CCfin$ to be whole of $\CC$.
Consequently, a precompact expansion $U : \CC^* \to \CC$ has the property that $U^{-1}(A)$ is finite for all $A \in \Ob(\CC)$.

\begin{THM}\label{srd-wap.thm.main}
  Let $\CC$ be a directed category of finite objects whose morphisms are mono.
  Then $\CC$ has finite Ramsey degrees if and only if
  there exists a category $\CC^*$ with the Ramsey property
  and a reasonable precompact expansion with unique restrictions $U : \CC^* \to \CC$ with the expansion property. (Such an expansion is usually referred to as a \emph{Ramsey expansion}.)
\end{THM}
\begin{proof}
  $(\Leftarrow)$ See \cite{masul-kpt}.

  $(\Rightarrow)$
  Since $\CC$ is a category of finite objects it has a skeleton $\SS$ with the properties listed above.
  
  \medskip

  Step 1. Let us first show that
  there exists a category $\SS^*$ with the Ramsey property
  and a reasonable expansion with unique restrictions $V : \SS^* \to \SS$ with the expansion property.

  The category $\SS$ is clearly directed, so the fact that it has finite Ramsey degrees
  ensures by Theorem~\ref{srd-wap.thm.srd=>wap} that it also has the weak amalgamation property.
  Since $\SS$ is trivially weakly dominated by itself, it follows that
  $\SS$ is a weak amalgamation category and it has a weak \Fraisse\ sequence $W = (W_n, w_n^m)_{n \le m \in \omega}$
  (Theorem~\ref{srd-wap.thm.kubis-weak-fraisse}~$(a)$).
  Let $\DD = \sigma_0\SS$ and let $\DDfin = J(\SS)$. Let us show that (C1) -- (C5) are satisfied.

  (C1) Morphisms in $\DD$ are mono by Lemma~\ref{srd-wap.lem.monos}

  (C2) $\Ob(\DDfin) = \Ob(J(\SS))$ is a set because $\Ob(\SS)$ is a set by the assumption (actually, a countable one).

  (C3) and (C5) follow from the fact that $\SS$ is a skeleton of a category of finite objects.

  (C4) take any $X = (X_n, x_n^m)_{n \le m \in \omega} \in \Ob(\DD)$ and note that there is a morphism
       $\overline{X_0} \to X$ given by
       \begin{center}
         \begin{tikzcd}
           X_0 \arrow[r, "\id"] \arrow[d, "\id"'] & X_0 \arrow[r, "\id"] \arrow[d, "x_0^1"'] & X_0 \arrow[r, "\id"] \arrow[d, "x_0^2"'] & \cdots\\
           X_0 \arrow[r, "x_0^1"'] & X_1 \arrow[r, "x_1^2"'] & X_2 \arrow[r, "x_2^3"'] & \cdots
         \end{tikzcd}
       \end{center}

  It is easy to see that $W$ as an object of $\DD$ is universal for $J(\SS)$. Moreover, $W$  is weakly homogeneous for $J(\SS)$
  by Theorem~\ref{srd-wap.thm.kubis-weak-fraisse}~$(c)$. Let us show that $W$ is locally finite for~$J(\SS)$.

  Take any $A, B \in \Ob(\SS)$ and arbitrary morphisms $(E, \epsilon) / \Boxed{\approx} \in \hom_\DD(\overline A, W)$,
  $(F, \phi) / \Boxed{\approx} \in \hom_\DD(\overline B, W)$. Without loss of generality we may assume that
  the transformations $(E, \epsilon)$ and $(F, \phi)$ are chosen so that $\epsilon(0) = \phi(0)$.
  Let $k = \epsilon(0) = \phi(0)$. Then the fact that every object in $\SS$ is locally finite for $\SS$ implies
  that there is a $D \in \Ob(\SS)$ and morphisms $p \in \hom_\SS(A, D)$, $q \in \hom_\SS(B, D)$ and
  $r \in \hom_\SS(D, W_k)$ such that
  \begin{center}
    \begin{tikzcd}
      D \arrow[rr, "r"] & & W_k & &  \\
      & A \arrow[ul, "p"] \arrow[ur, "E_0"' near start] & & B \arrow[ulll, "q" near start] \ar[ul, "F_0"']
    \end{tikzcd}
  \end{center}
  and $D$ is the ``smallest'' such object in the sense that for every other contender there is a morphism from $D$ into the contender
  such that everything commutes.

  Let $(G, \gamma)$ be the transformation $\overline D \to W$ there $\gamma : \omega \to \omega : n \mapsto n + k$
  and $G_n \in \hom_\SS(D, W_{n+k})$ is given by $G_n = w_k^{k+n} \cdot r$:
  \begin{center}
    \begin{tikzcd}
      & & D \arrow[r, "\id"] \arrow[d, "G_0=r"'] & D \arrow[r, "\id"] \arrow[d, "G_1"'] & D \arrow[r, "\id"] \arrow[d, "G_2"'] & \cdots\\
      W_0 \arrow[r] & \cdots \arrow[r] & W_k \arrow[r, "w_k^{k+1}"'] & W_{k+1} \arrow[r, "w_{k+1}^{k+2}"'] & W_{k+2} \arrow[r, "w_{k+2}^{k+3}"'] & \cdots
    \end{tikzcd}
  \end{center}
  Let us show that $(G, \gamma) \cdot \overline p \approx (E, \epsilon)$. Take any $n \in \omega$ and let
  $m = 1 + \max\{k + n, \epsilon(n)\}$. Note that $E_n = w_{k}^{\epsilon(n)} \cdot E_0$
  because $(E, \epsilon)$ is a transformation $\overline A \to W$ (recall that $k = \epsilon(0)$):
  \begin{center}
    \begin{tikzcd}
      A \arrow[rr, "\id"] \arrow[d, "E_0"'] & & A \arrow[d, "E_n"] \\
      W_{k} \arrow[rr, "w_{k}^{\epsilon(n)}"'] & & W_{\epsilon(n)}
    \end{tikzcd}
  \end{center}
  Therefore,
    $w_{\epsilon(n)}^m \cdot E_n
    = w_{\epsilon(n)}^m \cdot  w_{k}^{\epsilon(n)} \cdot E_0 = w_k^m \cdot E_0
    = w_{k+n}^m \cdot w_k^{k+n} \cdot r \cdot p
    = w_{k+n}^m \cdot G_n \cdot p$.
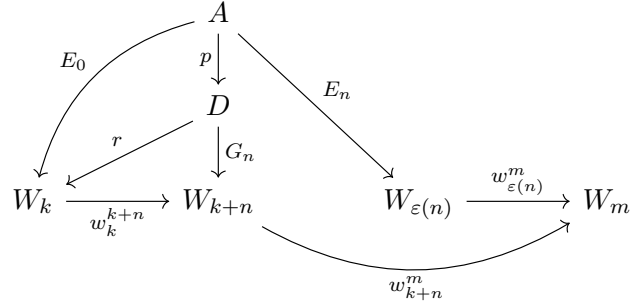
  \begin{center}
    \begin{tikzcd}[column sep=large]
      & A \arrow[ddl, bend right, "E_0"'] \arrow[d, "p"'] \arrow[ddr, "E_n"] & & \\
      & D \arrow[dl, "r"'] \arrow[d, "G_n"] & & \\
      W_k \arrow[r, "w_k^{k+n}"'] & W_{k+n} \arrow[rr, bend right, "w_{k+n}^m"'] & W_{\epsilon(n)} \arrow[r, "w_{\epsilon(n)}^m"] & W_m
    \end{tikzcd}
  \end{center}
  The proof of $(G, \gamma) \cdot \overline q \approx (F, \phi)$ is analogous

  To complete the proof that $W$ is locally finite for $J(\SS)$ assume now that there is a contender $\overline C \in J(\SS)$
  with morphisms $\overline u / \Boxed{\approx} : \overline A \to \overline C$, 
  $\overline v / \Boxed{\approx} : \overline B \to \overline C$ and $(H, \theta) / \Boxed{\approx} : \overline C \to W$,
  where $u \in \hom_\SS(A, C)$ and $v \in \hom_\SS(B, C)$. Without loss of generality we may assume that the transformation
  $(H, \theta)$ was chosen so that $\theta(0) = k$. Then we have the following configuration in~$\SS$:
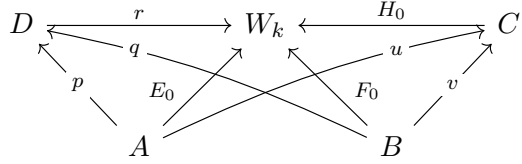
  \begin{center}
    \begin{tikzcd}[row sep=large]
      D \arrow[rr, "r"] & & W_k & &  C \arrow[ll, "H_0"'] \\
      & A \arrow[ul, "p" description] \arrow[ur, near start, "E_0"] \arrow[urrr, near end, bend left=7, "u" description] & & B \arrow[ulll, "q" description, near end, bend right=7] \ar[ul, near start, "F_0"'] \arrow[ur, "v"' description]
    \end{tikzcd}
  \end{center}
  By the choice of $D$ there is a morphism $s \in \hom_\SS(D, C)$ such that
  
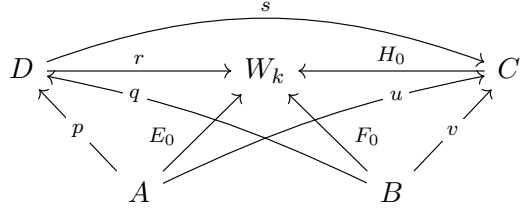
\begin{center}
    \begin{tikzcd}[row sep=large]
      D \arrow[rr, "r"] \arrow[rrrr, bend left=20, "s"] & & W_k & &  C \arrow[ll, "H_0"'] \\
      & A \arrow[ul, "p" description] \arrow[ur, near start, "E_0"] \arrow[urrr, near end, bend left=7, "u" description] & & B \arrow[ulll, "q" description, near end, bend right=7] \ar[ul, near start, "F_0"'] \arrow[ur, "v"' description]
    \end{tikzcd}
  \end{center}
  Note that everything still commutes when this diagram is transposed to $\DD$. The proof is straightforward. For example,
  $(H, \theta) \cdot \overline s \approx (G, \gamma)$ follows from
  
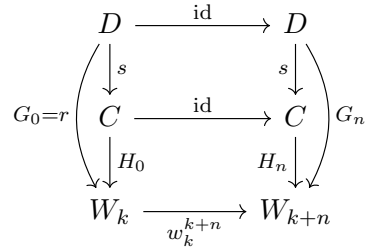
\begin{center}
    \begin{tikzcd}[column sep=large]
      D \arrow[r, "\id"] \arrow[d, "s"] \arrow[dd, bend right, "G_0=r"'] & D \arrow[d, "s"'] \arrow[dd, bend left, "G_n"]\\
      C \arrow[r, "\id"] \arrow[d, "H_0"] & C \arrow[d, "H_n"'] \\
      W_k \arrow[r, "w_k^{k+n}"'] & W_{k+n}
    \end{tikzcd}
  \end{center}

  Therefore, $\DD$ is a locally small category and, $\DDfin = J(\SS) \cong \SS$ is a full subcategory of $\DD$ and (C1) -- (C5) hold.
  Moreover, $W$ is a weakly homogeneous locally finite object and $\Age(F)$ is the whole of $\DDfin \cong \SS$.
  Since $\SS$ has finite Ramsey degrees, Theorem~\ref{crt.thm.akpt-main2} ensures that
  there is a reasonable precompact expansion with unique restrictions $V : \DD^* \to \DD$ and
  a full subcategory $\SS^*$ of $\DD^*$ which is directed, has the Ramsey property and
  $\restr{V}{\SS^*} : \SS^* \to \SS$ has the expansion property.

  \medskip

  Step 2. Since $\SS$ is a skeleton of $\CC$ we can now easily expand $\SS^*$ to $\CC^*$ and $V : \SS^* \to \SS$ to $U : \CC^* \to \CC$
  as follows. Let $F : \CC \to \SS$ be a functor that takes $C \in \Ob(\CC)$ to the unique $S \in \Ob(\SS)$ such that
  $C \cong S$. Next, for each $C \in \Ob(\CC)$ fix an isomorphism $\eta_C : C \to F(C)$ and define $F$ on morphisms so that $F$
  takes $f : C \to D$ to $F(f) : F(C) \to F(D)$ where $F(f) = \eta_D \cdot f \cdot \eta_C^{-1}$. This turns $\eta$ into a
  natural transformation $\ID \to F$. Now take $\CC^*$ to be the pullback of $\CC \overset F\longrightarrow \SS \overset V\longleftarrow \SS^*$
  in the quasicategory of all locally small categories and functors between them,
  and let $U : \CC^* \to \CC$ be the corresponding functor in the pullback diagram:
  \begin{center}
    \begin{tikzcd}
      \CC^* \arrow[r, "F^*"] \arrow[d, "U"'] & \SS^* \arrow[d, "V"] \\
      \CC \arrow[r,"F"'] & \SS
    \end{tikzcd}
  \end{center}
  It is easy to see that $\SS^*$ is a skeleton of $\CC^*$ and that $C^* \cong F^*(C^*)$ for every $C^* \in \Ob(\CC^*)$.
  Therefore, $\CC^*$ has the Ramsey property and $U : \CC^* \to \CC$
  is a reasonable expansion with unique restrictions and with the expansion property.
\end{proof}

The above result easily specializes to classes of first-order structures.
Let $\Theta$ be a first-order signature (that is, a set of functional, relational and constant symbols)
and let $\KK$ be a class of finite $\Theta$-structures. Then $\KK$ can be treated as a category with embeddings as morphisms.
So, when we stipulate that a class of first-order structures has certain properties
we have introduced in the context of categories (Ramsey property, finite Ramsey degrees, \ldots) we have this interpretation in mind.
Historically, structural Ramsey phenomena we consider in this paper were first identified in classes of
first-order structures and were later transposed to the context of abstract categories. For historical reasons we shall, therefore, say that
a class of first-order structures $\KK$ has the \emph{joint embedding property} if for all $\calA, \calB \in \KK$ there is a $\calC \in \KK$
such that $\calA \hookrightarrow \calC \hookleftarrow \calB$.
Clearly, $\KK$ has the joint embedding property if and only if $\KK$ is directed as a category.

Let $\Theta^* \supseteq \Theta$ be a first-order signature that contains $\Theta$ and let
$\KK^*$ be a class of $\Theta^*$-structures. Let $\KK$ be a class of $\Theta$-structures and let $U : \KK^* \to \KK$
be the forgetful functor that forgets the extra structure from $\Theta^* \setminus \Theta$ so that
$U$ takes a structure from $\KK^*$ to its $\Theta$-reduct (and takes homomorphisms to themselves).
We then say that $\KK^*$ is an \emph{expansion of} $\KK$, and that it is a \emph{reasonable expansion with unique restrictions
and with the expansion property} if $U : \KK^* \to \KK$ is.

\begin{COR}
  Let $\Theta$ be a first-order signature and let $\KK$ be a class of finite $\Theta$-structures such that
  there are at most countably many pairwise nonisomorphic structures in $\KK$ and $\KK$ has the joint-embedding property.
  Then: $\KK$ has finite Ramsey degrees if and only if there exists a first-order signature $\Theta^* \supseteq \Theta$
  and a class $\KK^*$ of $\Theta^*$-structures such that $\KK^*$ has the Ramsey property and
  $\KK^*$ is a reasonable precompact expansion of $\KK$ with unique restrictions and with the expansion property.
\end{COR}
\begin{proof}
  $(\Leftarrow)$: Immediate from Theorem~\ref{srd-wap.thm.main}.

  $(\Rightarrow)$: To show that this direction also follows from Theorem~\ref{srd-wap.thm.main} we have to show that
  the expansion constructed in Theorem~\ref{srd-wap.thm.main} can be performed in the realm of first-order structures.

  Let $\calA_m$, $m \in \NN$, be an enumeration of all representatives of isomorphism types in $\KK$ where
  $\calA_m$ is a $\Theta$-structure whose underlying set is $A_m = \{a_{m,1}, \ldots, a_{m, n_m}\}$, $n_m = |A_m|$.
  Clearly (with a slight abuse of terminology), $\AAA = \{\calA_m : m \in \NN\}$ is the skeleton of $\KK$.

  The first step in the proof of Theorem~\ref{srd-wap.thm.main} begins with the construction of the ambient category $\DD = \sigma_0\SS$ in which
  weakly homogeneous locally finite universal objects dwell. Since $\AAA$ is a set of finite $\Theta$-structures
  in the context of first-order structures we can take $\DD = \sigma_0\AAA$ to be the class of all structures isomorphic to the unions of countable chains in $\AAA$.
  Recall that the morphisms between sequences are defined so that a morphism from a sequence $X$ to a sequence $Y$
  corresponds uniquely to an embedding of the colimit of~$X$ into the colimit of~$Y$. Therefore, the ambient category $\DD$
  is a category whose objects are finite or countably infinite $\Theta$-structures which can be constructed as limits of
  countable chains in $\AAA$. Morphisms in $\DD$ are embeddings.

  Next, Theorem~\ref{crt.thm.akpt-main2} is invoked to produce a Ramsey expansion by adding additional structure
  to objects from $\DD$ so that for each $\calD \in \Ob(\DD)$ we add to $\Ob(\DD^*)$ all possible pairs
  $(\calD, (\delta_\calA)_{\calA \in \Ob(\AAA)})$ where $\delta_\calA : \hom(\calA, \calD) \to t_\AAA(\calA)$ is an arbitrary coloring.
  This particular construction specializes to first-order structures as follows (cf.~\cite{vanthe-finramdeg}).
  Let $t_\KK(\calA_m) = t_\AAA(\calA_m) = t_m \in \NN$ and let us expand $\Theta$ with countably many relational symbols $R_{m,j}$,
  $m \in \NN$, $1 \le j \le t_m$, where the arity of $R_{m,j}$ is $n_m = |A_m|$:
  $$
    \Theta' = \Theta \cup \{R_{m,j} : m \in \NN, 1 \le j \le t_m\}.
  $$
  Now, for each $\Theta$-structure $\calD = (D, \Theta^\calD) \in \Ob(\DD)$ we add to $\Ob(\DD^*)$ all possible $\Theta'$-structures
  $$
    \calD^* = (D, \Theta^\calD \cup \{R^{\calD^*}_{m,j} : m \in \NN, 1 \le j \le t_m\})
  $$
  where, for each $m \in \NN$:
  \begin{itemize}
    \item $R^{\calD^*}_{m,1}, \ldots, R^{\calD^*}_{m,t_{m}}$ are pairwise disjoint, some of them possibly empty;
    \item if $(d_1, \ldots, d_{n_m}) \in R^{\calD^*}_{m,j}$ then the map
          $\left(\begin{smallmatrix}a_{m,1} & \ldots & a_{m, n_m} \\ d_1 & \ldots & d_{n_m}\end{smallmatrix}\right)$
          is an embedding $\calA_m \hookrightarrow \calD$, $1 \le j \le t_m$; and
    \item if $\left(\begin{smallmatrix}a_{m,1} & \ldots & a_{m, n_m} \\ d_1 & \ldots & d_{n_m}\end{smallmatrix}\right)$
          is an embedding $\calA_m \hookrightarrow \calD$ then
          $(d_1, \ldots, d_{n_m}) \in R^{\calD^*}_{m,j}$ for some $1 \le j \le t_m$.
  \end{itemize}
  Morphisms in $\DD^*$ are embeddings.
  By Theorem~\ref{crt.thm.akpt-main2} the obvious forgetful functor $V : \DD^* \to \DD$ which takes a $\Theta'$ structure to its
  $\Theta$-reduct is a reasonable precompact expansion with unique restrictions and there is a full subcategory $\AAA^*$ of $\DD^*$ which is directed,
  has the Ramsey property and $\restr{V}{\AAA^*} : \AAA^* \to \AAA$ has the expansion property.

  The final step in the proof of Theorem~\ref{srd-wap.thm.main} consists of spreading the construction that we performed on $\AAA$ to
  the whole of $\KK$, and in the context of first-order structures this reduces to constructing isomorphic copies of elements of
  $\AAA^*$ simply by renaming elements.
\end{proof}

Moreover, the dual of Theorem~\ref{srd-wap.thm.main} holds as well.
We shall say that $\CC$ is a \emph{category of finite quotients} if $\CC$ has a skeleton $\SS$
with the following properties:
\begin{itemize}
  \item $\SS$ is a countable category,
  \item $\hom_\SS(A, B)$ is finite for all $A, B \in \SS$,
  \item for every $B \in \Ob(\SS)$ the set $\{A \in \Ob(\SS) : B \toSS A\}$ is finite, and
  \item every $A \in \Ob(\SS)$ is dually locally finite for $\SS$.
\end{itemize}
Here, $A \in \Ob(\SS)$ is \emph{dually locally finite for $\SS$} if $A \in \Ob(\SS)$ is locally finite for $\SS^\op$.
Continuing in the same fashion, we say that $\CC$ is \emph{dually directed} if $\CC^\op$ is directed;
that $\CC$ has \emph{small dual Ramsey degrees} if $\CC^\op$ has small Ramsey degrees;
that $\CC$ has \emph{dual Ramsey property} if $\CC^\op$ has the Ramsey property;
that an expansion $U : \CC \to \DD$ is \emph{dually reasonable} if $U : \CC^\op \to \DD^\op$ is reasonable;
that an expansion $U : \CC \to \DD$ has \emph{unique quotients} if $U : \CC^\op \to \DD^\op$ has unique restrictions; and
that an expansion $U : \CC \to \DD$ has \emph{the dual expansion property} if $U : \CC^\op \to \DD^\op$ has the expansion property.

\begin{COR}
  Let $\CC$ be a dually directed category of finite quotients whose morphisms are epi.
  Then $\CC$ has finite dual Ramsey degrees if and only if
  there exists a category $\CC^*$ with the dual Ramsey property
  and a dually reasonable precompact expansion with unique quotients $U : \CC^* \to \CC$ with the dual expansion property.
\end{COR}

\subsection*{Acknowledgements}

The authors would like to thank Adam Barto\v s for helping to streamline some arguments and Gianluca Basso for finding some mistakes in an earlier draft.

The first author was supported by the Science Fund of the Republic of Serbia, Grant No.~7750027:
Set-theoretic, model-theoretic and Ramsey-theoretic phenomena in mathematical structures: similarity and diversity -- SMART.


\end{document}